\documentclass[a4paper, 10pt]{ieeeconf}      % Use this line for a4
\IEEEoverridecommandlockouts                              % This command is only
\usepackage{times} % assumes new font selection scheme installed
\usepackage{amsmath} % assumes amsmath package installed
\usepackage{amssymb}  % assumes amsmath package installed

\newtheorem{definition}{Definition}
\newtheorem{proposition}{Proposition}
\newtheorem{theorem}{Theorem}
\newtheorem{lemma}{Lemma}
\newtheorem{corollary}{Corollary}
\newtheorem{remark}{Remark}
\newtheorem{example}{Example}

\newcommand{\defterm}[1]{\emph{#1}} % Defining a new term
\newcommand{\defabbr}[1]{\textbf{#1}} % Introduce a new abbreviation
\newcommand{\ubgec}{\mathcal{U}^\infty_{\text{UBgE}}} 
\newcommand{\ucccpi}{\mathcal{U}^\infty_{\text{UCC+}}}

\newcommand{\cK}{\mathcal{K}} % Class K
\newcommand{\cKi}{\mathcal{K}_\infty} % Class K infinity
\newcommand{\cKL}{\mathcal{KL}} % Class KL
\newcommand{\Rnn}{\mathbb{R}_{\geq 0}} % Nonnegative real numbers
\newcommand{\Rp}{\mathbb{R}_{>0}} % Positive real numbers

\title{\LARGE \bf
Revisiting the derivation of stage costs in infinite horizon discrete-time optimal control
}

\author{Christian Fiedler* and Sebastian Trimpe*% <-this % stops a space
\thanks{*Institute for Data Science in Mechanical Engineering, Faculty of Mechanical Engineering, RWTH Aachen University, Aachen, Germany
        {\tt\small \{christian.fiedler,trimpe\}@dsme.rwth-aachen.de}}%
}

\usepackage{fancyhdr}
\newcommand{\mytitle}{\textbf{Accepted version.} To appear in \emph{30th Mediterranean Control Conference 2022}.
\copyright 2022 IEEE. Personal use of this material is permitted. Permission
from IEEE must be obtained for all other uses, in any current or future
media, including reprinting/republishing this material for advertising or
promotional purposes, creating new collective works, for resale or
redistribution to servers or lists, or reuse of any copyrighted component of
this work in other works.}
\begin{document}

\maketitle
\thispagestyle{fancy} % PREPRINT
\pagestyle{empty}

\begin{abstract}
In many applications of optimal control, the stage cost is not fixed, but rather a design choice with considerable impact on the control performance. In infinite horizon optimal control, the choice of stage cost is often restricted by the requirement of uniform cost controllability, which is nontrivial to satisfy. Here we revisit a previously proposed constructive technique for stage cost design. We generalize its setting, weaken the required assumptions and add additional flexibility. Furthermore, we show that the required assumptions essentially cannot be weakened anymore. 
By providing improved design options for stage costs, this work contributes to expanding the applicability of optimization-based control methodologies, in particular, model predictive control.
\end{abstract}

\section{INTRODUCTION} \label{sec.intro}
Infinite horizon optimal control problems are at the core of many modern control approaches. For example, model predictive control (MPC) can be interpreted as a finite horizon approximation to a given infinite horizon optimal control problem \cite[Section~5.4]{GruenePannek17}, \cite{Gruene16}. An important component of an optimal control problem is the stage cost, the instantaneous cost of an applied control input and the current state of the system. In many control applications, the stage cost is not completely specified, but rather a design parameter. Examples include automatic controller tuning \cite{MarcoEtAl16} and MPC without terminal conditions \cite[Chapter~6]{GruenePannek17}. Especially in the latter case, goals like stabilization and constraint satisfaction are primary, and the concrete choice of a stage cost is secondary. The major technical challenge is then the design of a stage cost leading to a uniformly bounded value function of the corresponding infinite horizon optimal control problem, i.e., uniform cost controllability, cf. Section \ref{sec.controllability} for details.

A constructive method is described in \cite[Section~4.1]{GruenePannek17} that allows to derive a stage cost from certain uniform controllability assumptions, leading to uniform cost controllability. Due to the relevance of good stage cost designs, we revisit this technique in the present work. In particular, we use a more general setup and weaken the necessary controllability assumptions. Furthermore, we allow for more flexibility in the design of the stage cost and cover explicitly the case of interaction terms between state and control, which seems to have not been discussed so far. Finally, by providing a converse result, we show that the controllability assumptions essentially cannot be further weakened in general, an observation that has not been explicitly made before, to the best of our knowledge.

\textbf{Outline} In Section \ref{sec.preliminaries}, we introduce the problem setting and provide some necessary technical background. The relevant controllability assumptions are introduced and investigated in Section \ref{sec.controllability} and then used in Section \ref{sec.stagecosts} to derive suitable stage costs. We demonstrate in Section \ref{sec.converse} that the utilized controllability assumptions are indeed reasonable in this context. Section \ref{sec.conclusion} concludes with some discussion and directions for future work.

\section{PRELIMINARIES} \label{sec.preliminaries}
We consider discrete-time nonlinear control systems
\begin{equation} \label{eq.sys}
x_+ = f(x,u)
\end{equation}
described by a transition function $f: X \times U \rightarrow X$. 
The state space $X$ and the set of control inputs $U$ can be arbitrary nonempty sets. 
As usual, for a given input sequence $u$, the resulting trajectory is defined by
\begin{align*}
\phi(0;x,u) & = x\\
\phi(n+1;x,u) & = f(\phi(n;x,u),u(n)),\:n\geq 0.
\end{align*}
Similar to for example \cite{GrimmEtAl05} or \cite{TranKellettDower15}, we work in a very general setting using a \emph{state measure} $\sigma: X \rightarrow \mathbb{R}_{\geq 0}$ and an \emph{input} or \emph{control measure} $\rho: U \rightarrow \mathbb{R}_{\geq 0}$. For the following developments, we do not require any special properties of $\sigma$ and $\rho$.
\begin{example} \label{ex.statecontrolmeasures}
This general setup covers most relevant settings, including the following cases:

\emph{1)} $(X,d_X)$ a metric space, $A \subseteq X$ a nonempty, compact subset, and
	\begin{equation*}
		\sigma(x) = \inf_{x^\prime \in A} d_X(x,x^\prime),
	\end{equation*}
	i.e., $\sigma(x)$ is the distance of $x$ to the set $A$. A very important special case is $A=\{ x_\ast \}$ for some $x_\ast \in X$, where $\sigma(x)=d_X(x,x_\ast)$.

\emph{2)} $(X,d_X)$ and $(U,d_U)$ metric spaces,  $(x_\ast,u_\ast)\in X \times U$ an equilibrium pair, i.e., $x_\ast=f(x_\ast,u_\ast)$, and
	\begin{equation*}
		\sigma(x) = d_X(x,x_\ast), \qquad \rho(u) = d_U(u,u_\ast).
	\end{equation*}
	This is essentially the (time-invariant version of the) setting in \cite[Section~4.1]{GruenePannek17}.

\emph{3)} Probably the most common case is $X=\mathbb{R}^n$, $U=\mathbb{R}^m$, $(0,0) \in \mathbb{R}^n \times \mathbb{R}^m$ an equilibrium and
	\begin{equation*}
		\sigma(x)  = \|x\|_2, \qquad \rho(u) = \|u\|_2.
	\end{equation*}
\end{example}

For a stage cost $\ell: X \times U \rightarrow \Rnn$ and horizon length $N \in \mathbb{N}_+ \cup \{ \infty \}$, define the total cost for initial state $x \in X$ and input sequence $u \in U^N$ (setting $U^\infty=U^{\mathbb{N}}$) by
\begin{equation}
J_N(x,u \mid \ell) = \sum_{n=0}^{N-1} \ell(\phi(n;x,u), u(n)).
\end{equation}

Finally, we use extensively the popular comparison function formalism. 
The class $\cK$ consists of all functions $\alpha : \Rnn \rightarrow \Rnn$ with $\alpha(0)=0$ and that are continuous and strictly increasing. 
The class $\cKi$ consists of functions $\alpha \in \cK$ such that $\lim_{r \rightarrow \infty} \alpha(r) = \infty$ and the class $\cKL$ consists\footnote{Unfortunately the definition of this class is not uniform in the literature, but most versions are compatible with the results that follow.} of continuous functions $\beta: \Rnn \times \Rnn \rightarrow \Rnn$ such that for all $t \in \Rnn$ we have $\beta(\cdot,t) \in \cKi$ and for all $r \in \Rp$ the function $\beta(r,\cdot)$ is strictly decreasing and $\lim_{t \rightarrow \infty} \beta(r,t) = 0$.
To improve readibility, we often apply pointwise unary and binary operations as well as relations to comparison functions without making the argument explicit. For example, for $\alpha_1,\alpha_2 \in \cKi$, the expression $\alpha_1 \leq \alpha_2$ means $\forall r  \in \Rnn: \alpha_1(r) \leq \alpha_2(r)$.

The next lemma collects some properties of comparison functions that are needed in the remainder of this work. For proofs and further background, see for example \cite{Kellett14}.
\begin{lemma} 
\begin{enumerate}
\item All $\alpha \in \cKi$ are invertible with $\alpha^{-1} \in \cKi$.
\item For all $\alpha_1,\alpha_2 \in \cKi$, we also have $\alpha_1 \cdot \alpha_2 \in \cKi$ and $\alpha_1 \circ \alpha_2 \in \cKi$.
\item For all $\alpha_1,\alpha_2 \in \cKi$ and $C_1,C_2 \in\Rp$ we have $C_1\alpha_1 + C_2\alpha_2 \in \cKi$.
\item (Weak triangle inequality) For all $\alpha\in\cKi$ and $a,b\in\Rnn$, we have $\alpha(a+b)\leq \alpha(2a) + \alpha(2b)$.
\item (Sontag's $\cKL$-Lemma) Let $\beta \in \cKL$ and $\theta \in (0,1)$ be arbitrary. Then there exist $\gamma_1,\gamma_2\in\cKi$ such that for all $r,t \in \Rnn$ we have $\beta(r,t)\leq \gamma_2(\theta^t\gamma_1(r))$. We call this a $\cKL$-decomposition in the following.
\end{enumerate}
\end{lemma}
\section{ASYMPTOTIC AND COST CONTROLLABILITY CONDITIONS} \label{sec.controllability}
We now summarize the controllability notions used in the sequel. Note that we use some slightly nonstandard, but descriptive terms.
\begin{definition} \label{def.uac}
Let $\mathbb{X} \subseteq X$. We say that system \eqref{eq.sys} is
\begin{enumerate}
\item  \defterm{uniformly asymptotically controllable} (\defabbr{UAC}) on $\mathbb{X}$ if there exists $\beta \in \cKL$ such that for all $x \in \mathbb{X}$ there exists $u_x \in U^\infty$ with
\begin{equation}
\sigma(\phi(n;x,u_x)) \leq \beta(\sigma(x),n) \: \forall n\geq 0
\end{equation}
\item \defterm{uniformly asymptotically controllable with uniformly vanishing controls} (\defabbr{UACwUVC}) on $\mathbb{X}$ if there exist $\beta_x,\beta_u \in \cKL$ such that for all $x \in \mathbb{X}$ there exists $u_x \in U^\infty$ with
\begin{align}
\sigma(\phi(n;x,u_x)) & \leq \beta_x(\sigma(x),n) \: \forall n\geq 0 \\
\rho(u_x(n)) & \leq \beta_u(\sigma(x),n) \: \forall n\geq 0 \label{eq.uacuvc.2}
\end{align}
\item \defterm{uniformly asymptotically controllable with uniformly bounded generalized energy controls} (\defabbr{UACwUBgEC}) on $\mathbb{X}$ if there exist $\beta \in \cKL$, $\eta, \gamma \in \cKi$ such that
$\forall x \in \mathbb{X}$ there exists $u_x \in U^\infty$ with
\begin{align} \label{eq.uacube.1}
&\sigma(\phi(n;x,u_x))) \leq \beta(\sigma(x), n) \: \forall n \geq 0\\
&\sum_{n=0}^{N-1} \eta(\rho(u_x(n))) \leq \gamma(\sigma(x)) \: \forall N \geq 1 \label{eq.uacube.2}
\end{align}
For convenience, denote the set of all such control sequences by $\ubgec(x)$.
\end{enumerate}
\end{definition}
\begin{remark}
If the system is UACwUVC, then for all $w_1,w_2 \in \Rnn$ there exists $\beta \in \mathcal{KL}$ such that for all $x \in \mathbb{X}$ there exists $u_x \in U^\infty$ with
\begin{equation*} \forall n \geq 0: 
w_1 \sigma(x(n;x,u_x)) + w_2\rho(u_x(n)) \leq \beta(\sigma(x),n)
\end{equation*}
(simply use $\beta(r,t)=\max\{w_1,1\} \beta_x(r,t) + \max\{w_2,1\}\beta_u(r,t)$). On the other hand, if there are $w_1,w_2 \in \Rp$ and $\beta \in \mathcal{KL}$ such that for all $x \in \mathbb{X}$ there exists $u_x \in U^\infty$ with
% TYPO
% x --> \phi
% Corrected 23/05/22
\begin{equation*}
\forall n \geq 0: \:
w_1 \sigma(\phi(n;x,u_x)) + w_2\rho(u_x(n)) \leq \beta(\sigma(x),n),
\end{equation*}
then the system is UACwUVC (use $\beta_x(r,t)=\beta_u(r,t)=\max\{1/w_1,1/w_2\}\beta(r,t)$).
This shows that in the setting of Example \ref{ex.statecontrolmeasures}, item 2), the definition of UAC and UACwUVC reduces to \emph{uniform controllability with the small control property} as described in \cite[Definition~4.1]{GruenePannek17}. 
\end{remark}
\begin{remark}
The definition of UACwUBgEC is motivated by the case $\eta(r)=r^2$ which can be interpreted as physical energy. Note that we focus on $\eta \in \mathcal{K}_\infty$ since this aligns closest with the interpretation of generalized energy. Conceptually, our definition is similar to \emph{summability} properties, cf. e.g. \cite{TeelEtAl10}, but we are not aware of previous usages in the context of asymptotic controllability.
\end{remark}
Next, we clarify some relations between the different concepts from Definition \ref{def.uac}.
\begin{proposition} \label{prop.UVAimpliesUFgEC}
\begin{enumerate}
\item If system \eqref{eq.sys} is UACwUVC on $\mathbb{X}$, then it is also UACwUBgEC on $\mathbb{X}$.
\item If system \eqref{eq.sys} is UACwUBgEC on $\mathbb{X}$, then it is UAC.
\end{enumerate}
\end{proposition}
\begin{proof}
%UACwUVC implies UACwUBgEC: 
1) For $\beta\in\cKL$ in \eqref{eq.uacube.1} we can choose $\beta=\beta_x$.
For \eqref{eq.uacube.2}, we use the technique from \cite[Theorem~4.3]{GruenePannek17}: 
For $\beta_u$ from the definition of UACwUVC, let $\beta_u(r,n) \leq \gamma_2(\theta^n \gamma_1(r))$ be any $\cKL$-decomposition, and choose any $\eta \in \cK_\infty$ with $\eta \leq \gamma_2^{-1}$ ($\eta=\gamma_2^{-1}$ works for example), resulting in
\begin{equation*}
\eta(\beta_u(r,n)) \leq \eta\big(\gamma_2(\theta^n \gamma_1(r))\big) \leq \theta^n \gamma_1(r)  \: \forall r \geq 0,\forall n \in \mathbb{N}.
\end{equation*}
Let $x \in \mathbb{X}$ be arbitrary and choose a $u_x\in U^\infty$ fulfilling \eqref{eq.uacuvc.2}, then we have for $N\geq 1$
\begin{align*}
\sum_{n=0}^{N-1} \eta(\rho(u(n)))
& \leq \sum_{n=0}^{N-1} \theta^n \gamma_1(\sigma(x)) \leq \frac{1}{1-\theta} \gamma_1(\sigma(x)),
\end{align*}
hence UACwUBgEC with $\gamma= \frac{1}{1-\theta} \gamma_1$.

%UACwUBgEC implies UAC: 
2) Clear from the definitions.
\end{proof}
In particular, Proposition \ref{prop.UVAimpliesUFgEC} shows that the UACwUVC property is at least as strong as UACwUBgEC.
\begin{remark}
We almost have a converse of the first implication in the preceding proposition. If the system is UACwUBgEC, then
$\sum_{n=0}^\infty \eta\big(\rho(u(n))\big) \leq \gamma(\sigma(x))<\infty$ and $\eta\in\cKi$ imply that $\rho(u(n))\rightarrow 0$ for $n \rightarrow \infty$, 
but the convergence rate might depend on $x$ and $u_x$, and hence we might not have uniformity.
We conjecture that under mild compactness and continuity assumptions we can ensure this uniformity and hence UACwUBgEC and UACwUVC become equivalent.
\end{remark}
\begin{definition}
We say that system \eqref{eq.sys} is \defterm{uniformly cost controllable} (\defabbr{UCC}) on $\mathbb{X}\subseteq X$ with respect to (w.r.t.) the stage cost $\ell: X \times U \rightarrow \Rnn$ and total cost bound $\bar{\alpha}\in\cKi$ if for all $x \in \mathbb{X}$ there exists $u_x \in U^\infty$ such that
\begin{equation}
J_\infty(x,u_x \mid \ell) \leq \bar{\alpha}(\sigma(x)).
\end{equation}
\phantom{.}
\end{definition}
The UCC property is sometimes referred to as \defterm{weak controllability}, see for example \cite[Assumption~2.17]{RawlingsMayneDiehl17}.
\section{DERIVING STAGE COSTS FROM ASYMPTOTIC CONTROLLABILITY} \label{sec.stagecosts}
The starting point is the derivation of stage costs of the form $\ell(x,u)=q(\sigma(x))+r(\rho(u))$ for $q \in \cKi$ and some $r:\Rnn\rightarrow\Rnn$. 
The following result is a direct adaption of \cite[Theorem~4.3]{GruenePannek17} to the present setting, allowing for additional flexibility in the choice of $q$ and $r$ and requiring only UACwUBgEC instead of UACwUVC.
\begin{theorem}\label{thm.stagecostFromUAC}
Let $\mathbb{X} \subseteq X$, $\beta \in \cKL$, $\eta: \Rnn \rightarrow \Rnn$ some function, $\gamma \in \cKi$, such that for all $x \in \mathbb{X}$ there exists $u_x \in U^\infty$ with
\begin{enumerate}
	\item $ \forall n \geq 0:\: \sigma\big(\phi(n;x,u_x)\big) \leq \beta(\sigma(x), n)$
	\item $\forall N \geq 1: \: \sum_{n=0}^{N-1} \eta\big(\rho(u_x(n))\big) \leq \gamma(\sigma(x))$
\end{enumerate}
Then there exists $q \in \cKi$, $r: \Rnn \rightarrow \Rnn$ and $\bar{\alpha} \in \cKi$ such that system \eqref{eq.sys} is uniformly cost controllable on $\mathbb{X}$ w.r.t. $\ell(x,u)=q(\sigma(x)) + r(\rho(u))$ and total cost bound $\bar{\alpha}$. Additionally,
\begin{enumerate}
\item if $\eta$ is positive definite, then we can choose $r$ to be positive definite.
\item if $\eta \in \cKi$, i.e., the system is UACwUBgEC, then we can choose $r \in \cKi$.
\end{enumerate}

More explicitly, for any $\cKL$-decomposition $\beta(r,n)\leq \gamma_2(\theta^n\gamma_1(r))$ $\forall r \geq 0$, $n\in\mathbb{N}$, 
all $q \in \cKi$ with $q \leq C_q \cdot \gamma_2^{-1}$ for some $C_q >0$ ($q=\gamma_2^{-1}$ works), 
all $r: \Rnn \rightarrow \Rnn$ with $r \leq C_r \cdot \eta$ for some $C_r >0$ ($r=\eta$ works),
we have UCC w.r.t. $\ell(x,u)=q(\sigma(x)) + r(\rho(u))$ with $\bar{\alpha}=\frac{C_q}{1-\theta}\gamma_1 + C_r \cdot \gamma$.
\end{theorem}
\begin{proof}
Let $x \in \mathbb{X}$ be arbitrary and choose a corresponding $u_x \in U^\infty$.
For any $n\geq0$ we then have
\begin{align*}
q\big(\sigma(\phi(n;x,u_x))\big) & \leq q\big(\beta(\sigma(x),n)\big) \\
& \leq q\Big(\gamma_2\big(\theta^n\gamma_1(\sigma(x))\big)\Big) \\
& \leq C_q \theta^n \gamma_1(\sigma(x)),
\end{align*}
where we used the monotonicity of $q$ together with the $\cKL$-decomposition of $\beta$ and the choice $q \leq C_q \cdot \gamma_2^{-1}$.
For $N\geq 1$ we can continue with
\begin{align*}
J_N(x, u_x \mid \ell) & = \sum_{n=0}^{N-1} \ell(\phi(n;x,u_x), u_x(n)) \\
& =  \sum_{n=0}^{N-1} q\big(\sigma(x(n;x,u_x))\big) + r\big(\rho( u_x(n))\big) \\
& \leq \sum_{n=0}^{N-1} C_q \theta^n \gamma_1(\sigma(x)) + C_r \sum_{n=0}^{N-1} \eta\big(\rho(u_x(n))\big) \\
& \leq \frac{C_q}{1-\theta}\gamma_1(\sigma(x)) + C_r \gamma(\sigma(x)),
\end{align*}
where we used additionally the choice $r \leq C_r \cdot \eta$ and \eqref{eq.uacube.2}.
Since the last term is independent of $N$ (together with nonnegativity of $\ell$) we get convergence of the total cost together with the bound
\begin{equation*}
J_\infty(x,u_x \mid \ell) \leq \left( \frac{C_q}{1-\theta}\gamma_1 + C_r \gamma\right)(\sigma(x)) =: \bar{\alpha}(\sigma(x)).
\end{equation*}
The remaining claims are clear.
\end{proof}
\begin{remark}
The existence of $\underline{\alpha}\in\cKi$ with $\ell(x,u)\geq\underline{\alpha}(\sigma(x))$ for all $(x,u)$ is crucial for stabilization tasks, cf. \cite[Chapter~4]{GruenePannek17}. Since $q\in\cKi$, this is trivially satisfied here (actually, $q$ could be replaced by any $\tilde{q}\leq q$ with $\tilde{q}\geq\underline{\alpha}$ for some $\underline{\alpha}\in\cKi$).
\end{remark}
%
% Interaction terms
%
We now investigate which interaction terms between $x$ and $u$ might be possible while still ensuring UCC.
The following simple result already allows a considerable amount of flexibility.
\begin{proposition} \label{prop.simpleinteractionterms}
Consider the situation of Theorem \ref{thm.stagecostFromUAC} for the case of UACwUBgEC, and construct $q$ and $r$ as described there.
Let $s:X \times U \rightarrow \mathbb{R}_{\geq 0}$ be any function such that there exists $C_1,C_2,C_3 \in \Rnn$ and $\alpha\in\mathcal{K}_\infty$ such that for all $(x,u)\in X\times U$
\begin{equation*}
s(x,u)\leq C_1\gamma_2^{-1}(\sigma(x)) + C_2 \eta(\rho(u)) + C_3 \gamma_2^{-1}(\sigma(x))\alpha(\rho(u)).
\end{equation*}
Then $\mathbb{X}$ is UCC w.r.t. $\ell(x,u) = q(\sigma(x)) + r(\rho(u)) + s(x,u)$ and total cost bound
\begin{equation*}
\bar{\alpha} = \frac{C_q+C_1}{1-\theta}\gamma_1 + (C_r+C_2) \cdot \gamma + \frac{C_3}{1-\theta}\gamma_1 \cdot (\alpha \circ \eta^{-1} \circ \gamma)
\end{equation*}
\end{proposition}
\begin{proof}
Let $x \in \mathbb{X}$ be arbitrary, choose $u_x \in \ubgec(x)$ and define for readability 
$\sigma(n)=\sigma\big(\phi(n;x,u_x)\big)$ and $\rho(n)=\rho\big(u_x(n)\big)$. 
For all $n\geq 0$ we then have
\begin{align*}
s(\phi(n;x,u_x),u_x(n)) & \leq C_1 \gamma_2^{-1}(\sigma(n)) + C_2\eta(\rho(n)) \\
	& \hspace{0.5cm} +  C_3\gamma_2^{-1}(\sigma(n))\alpha(\rho(n)).
\end{align*}
Continuing with
\begin{align*}
\gamma_2^{-1}(\sigma(n)) & \leq \gamma_2^{-1}(\beta(\sigma(x),n)) \leq \theta^n \gamma_1(\sigma(x)),
\end{align*}
where we used the $\cKL$-decomposition from Theorem \ref{thm.stagecostFromUAC}, and
\begin{align*}
\alpha(\rho(n)) & = (\alpha \circ \eta^{-1})\big(\eta(\rho(n))\big) \leq (\alpha \circ \eta^{-1} \circ \gamma)(\sigma(x)),
\end{align*}
where we used again \eqref{eq.uacube.2}, 
leads to
\begin{align*}
s(\phi(n;x,u_x),u_x(n)) & \leq C_1\theta^n \gamma_1(\sigma(x)) + C_2 \eta(\rho(n)) \\
	& \hspace{-0.5cm} + C_3\theta^n\gamma_1(\sigma(x))\cdot(\alpha \circ \eta^{-1} \circ \gamma)(\sigma(x))
\end{align*}
Summing over $n$ yields
\begin{align*}
\sum_{n=0}^{N-1} s(\phi(n;x,u_x), u_x(n)) & \leq \frac{C_1}{1-\theta}\gamma_1(\sigma(x)) + C_2\gamma(\sigma(x)) \\
	& \hspace{-0.8cm} + \frac{C_3}{1-\theta}\gamma_1(\sigma(x)) \cdot (\alpha \circ \eta^{-1} \circ \gamma)(\sigma(x))
\end{align*}
and the claim follows by combining this with the results from Theorem \ref{thm.stagecostFromUAC}.
\end{proof}
%
%We now provide some examples of possible interaction terms $s$.
%
\begin{example} \label{ex.simpleinteractionterms}
Consider the situation of Theorem \ref{thm.stagecostFromUAC} for the case of UACwUBgEC. The goal is to find functions $s$ compatible with Proposition \ref{prop.simpleinteractionterms}.

\emph{1)} Any function $s(x,u)=\alpha_1(\sigma(x))\cdot\alpha_2(\rho(u))$ with $\alpha_1,\alpha_2\in\mathcal{K}_\infty$ works,  if $\alpha_1 \leq C \cdot \gamma_2^{-1}$ for some $C>0$. 
In this case we can set (using the notation of Proposition \ref{prop.simpleinteractionterms}) $C_1=C_2=0$, $C_3=C$, $\alpha=\alpha_2$, which results in the total cost bound
\begin{equation*}
\bar{\alpha} = \frac{C_q}{1-\theta} \gamma_1 + C_r \gamma + \frac{C}{1-\theta} \gamma_1 \cdot \left(\alpha_2 \circ \eta^{-1} \circ \gamma\right).
\end{equation*}

\emph{2)} Let $s(x,u) = \alpha_3\big(\alpha_1(\sigma(x)) + \alpha_2(\rho(u))\big)$ with $\alpha_1,\alpha_2,\alpha_3 \in \mathcal{K}_\infty$ such that
\begin{enumerate}
\item $\exists C_1 >0$ such that $\alpha_3 \circ (2\alpha_1) \leq C_1 \gamma_2^{-1}$
\item $\exists C_2>0$ such that $\alpha_3 \circ (2\alpha_2) \leq C_2 \eta$
\end{enumerate}
For example, we can choose arbitrary $\alpha_1,\alpha_2\in\mathcal{K}_\infty$ and then set
\begin{equation*}
\alpha_3(s) = \min\{ \gamma_2^{-1}(\alpha_1^{-1}(1/2s)), \eta(\alpha_2^{-1}(1/2s)) \}.
\end{equation*}
Using the weak triangle inequality for $\mathcal{K}_\infty$-functions, we get
\begin{align*}
s(x,u) & \leq \left( \alpha_3 \circ 2\alpha_1 \right)(\sigma(x)) + \left( \alpha_3 \circ 2\alpha_2 \right)(\rho(u)) \\
& \leq C_1 \gamma_2^{-1}(\sigma(x)) + C_2 \eta(\rho(u))
\end{align*}
and hence the total cost bound
\begin{equation*}
\bar{\alpha} = \frac{C_q + C_1}{1-\theta}\gamma_1 + (C_r + C_2)\gamma.
\end{equation*}
\end{example}
The next result allows an even greater degree of flexibility in designing stage costs that lead to UCC.
It is hinted at in the discussion after \cite[Theorem~4.3]{GruenePannek17}, however, the precise details are not given there. Furthermore, we state the result in the more general case of UACwUBgEC and using the state and input measures $\sigma$ and $\rho$.
\begin{theorem}
Let system \eqref{eq.sys} be UACwUBgEC on $\mathbb{X}$ and $\ell: X \times U \rightarrow \Rnn$ any stage cost such that
there are $R_\sigma \in \Rp$, $C_1,C_2,C_3 \in \Rnn$, $\alpha \in \cKi$ and $\alpha_x,\alpha_u \in \cKi$ such that
for all $x \in \mathbb{X}$ with $\sigma(x) < R_\sigma$ and all $u \in U$
\begin{equation}
\ell(x,u) \leq C_1\gamma_2^{-1}(\sigma(x)) + C_2 \eta(\rho(u)) + C_3\gamma_2^{-1}(\sigma(x))\alpha(\rho(u))
\end{equation}
and 
for all $x \in \mathbb{X}$ with $\sigma(x)\geq R_\sigma$ and all $u \in U$
\begin{equation} \label{eq.interactiontransientbound}
\ell(x,u) \leq \alpha_x(\sigma(x)) + \alpha_u(\rho(u)).
\end{equation}
Then system \eqref{eq.sys} is UCC on $\mathbb{X}$ w.r.t. $\ell$.
\end{theorem}
\begin{proof}
For each $R \in \Rnn$ there exists $N(R) \in \mathbb{N}$ such that $\beta(R,n) < R_\sigma$ for all $n\geq N(R)$.
We can choose the function $N: \Rnn \rightarrow \mathbb{N}$ to be nondecreasing.
Let $x \in \mathbb{X}$ be arbitrary, select $u_x \in \ubgec(x)$ and to avoid notational clutter, define $\ell(n)=\ell(\phi(n;x,u_x),u_x(n))$, $\sigma(n)=\sigma(\phi(n;x,u_x))$ and $\rho(n)=\rho(u_x(n))$.
Let $N \geq 1$ be arbitrary and define
\begin{align*}
\mathcal{I}_1 & = \{ n \in \mathbb{N} \mid \sigma(\phi(n;x,u_x)) \geq R_\sigma,\: n < N\} \\
\mathcal{I}_2 & = \{ n \in \mathbb{N} \mid \sigma(\phi(n;x,u_x)) < R_\sigma,\: n < N\}.
\end{align*}
By construction, the set $\mathcal{I}_1$ has at most $N(\sigma(x))$ elements and we have the partition
\begin{align*}
\sum_{n=0}^{N-1} \ell(n) = \sum_{n \in \mathcal{I}_1} \ell(n) + \sum_{n \in \mathcal{I}_2} \ell(n).
\end{align*}
For $n \in \mathcal{I}_1$ we find (since $\sigma(n) \geq R_\sigma$)
\begin{align*}
\ell(n) & \leq \alpha_x(\sigma(n)) + \alpha_u(\rho(n)) \\
& \leq \alpha_x( \beta(\sigma(x),n)) + (\alpha_u \circ \eta^{-1})\big(\eta(\rho(n))\big) \\
& \leq \alpha_x(\beta(\sigma(x),0)) + (\alpha_u \circ \eta^{-1} \circ \gamma)(\sigma(x)),
\end{align*}
hence
\begin{equation*}
\sum_{n\in\mathcal{I}_1} \ell(n) \leq \tilde{\alpha}_1(\sigma(x))
\end{equation*}
with
\begin{equation*}
\tilde{\alpha}_1(r) = N(r)\left(\alpha_x(\beta(r,0)) + (\alpha_u \circ \eta^{-1} \circ \gamma)(r)\right).
\end{equation*}
Since $\sigma(n) < R_\sigma$ for all $n \in \mathcal{I}_2$, we get from the proof of Proposition \ref{prop.simpleinteractionterms}
\begin{align*}
\sum_{n\in\mathcal{I}_2} \ell(n)&  \leq \frac{C_1}{1-\theta}\gamma_1(\sigma(x)) + C_2\gamma(\sigma(x)) \\
	& \hspace{0.5cm} + \frac{C_3}{1-\theta}\gamma_1(\sigma(x))(\alpha \circ \eta^{-1} \circ \gamma)(\sigma(x)) \\
	& =: \tilde{\alpha}_2(\sigma(x)).
\end{align*}
Altogether, we get
\begin{equation*}
\sum_{n=0}^{N-1} \ell(n) \leq \tilde{\alpha}_1(\sigma(x)) + \tilde{\alpha}_2(\sigma(x)).
\end{equation*}
Using standard arguments we can now upper bound $\tilde{\alpha}_1 + \tilde{\alpha}_2$ by a function $\bar{\alpha}\in\cKi$.
\end{proof}
\begin{remark}
\emph{1)} The bound \eqref{eq.interactiontransientbound} is only relevant for states with large state measure $\sigma$, corresponding to the transient behavior in the context of stabilization. Therefore, a rough bound of this form is enough. In particular, using the additive form considered here does not lead to a loss in generality. 
For example, assume we have a bound $\ell(x,u)\leq B(\sigma(x),\rho(u))$ in \eqref{eq.interactiontransientbound} instead, where $B: \Rnn \times \Rnn \rightarrow \Rnn$ is continuous and nondecreasing in the natural partial order on $\Rnn^2$ (i.e., $t_i \leq t_i^\prime$, $i=1,2$, implies $B(t_1,t_2)\leq B(t_1^\prime,t_2^\prime)$). Using arguments similar to \cite[Lemma~33]{Kellett14}, one can then use a bound of the form \eqref{eq.interactiontransientbound} again.

\emph{2)} In the bound \eqref{eq.interactiontransientbound}, the behaviour of $\alpha_x\lvert_{[0,R_\sigma)}$ does not matter. In particular, this justifies using only the term $\alpha_u(\rho(u))$ to include penalization of $\rho(u)$. For example, if we want to penalize $\rho(u)=0$ if $\sigma(x)\geq R_\sigma$, then this can be included in the term $\alpha_x$.
\end{remark}
\section{A CONVERSE RESULT} \label{sec.converse}
In order to allow flexibility in the design of stage costs, it is desirable that the techniques described in the previous section yield a large range of stage costs. In particular, the question arises whether the controllability conditions used to derive stage costs are much stronger than UCC with respect to a stage cost of a certain type. Since under mild assumptions a stage cost as constructed in Theorem \ref{thm.stagecostFromUAC} can be used to derive a stabilizing controller, cf. \cite[Section~4.3]{GruenePannek17}, it is clear that UCC implies some form of controllability. The next result provides, under an additional controlled-invariance assumption, a converse to Theorem \ref{thm.stagecostFromUAC}. 
\begin{theorem} \label{thm.converse}
Let $\mathbb{X}\subseteq X$, $\ell(x,u)=q(\sigma(x))+r(\rho(u))$ and $\bar{\alpha}\in\cKi$, where $q\in\cKi$ and $r:\Rnn\rightarrow\Rnn$, such that for all $x \in \mathbb{X}$ there exists $u \in U^\infty$ with
\begin{enumerate}
\item $\forall n\geq 0: \phi(n;x,u)\in\mathbb{X}$
\item $J_\infty(x,u\mid \ell) \leq \bar{\alpha}(\sigma(x))$
\end{enumerate}
Denote the set of all such $u$ by $\ucccpi(x)$.
Then there exists $\beta \in \cKL$, $\eta:\Rnn\rightarrow\Rnn$ and $\gamma\in\cKi$ such that for all $x \in \mathbb{X}$ there exists $u_x\in U^\infty$ with
\begin{enumerate}
\item $\sigma(\phi(n;x,u_x))\leq \beta(\sigma(x),n),\: \forall n\geq 0$
\item $\sum_{n=0}^{N-1} \eta\big(\rho(u_x(n))\big) \leq \gamma(\sigma(x)),\:\forall N \geq 1$
\end{enumerate}
If $r \in \cKi$, then we can choose $\eta \in \cKi$, i.e., the system is UACwUBgEC on $\mathbb{X}$.
\end{theorem}
\begin{proof}
\textbf{Step 1} For all $x \in \mathbb{X}$ and $u \in \ucccpi(x)$ we have for all $n\geq 0$
\begin{equation*}
\sigma(\phi(n;x,u)) \leq \gamma_\sigma(\sigma(x))
\end{equation*}
for $\gamma_\sigma = q^{-1} \circ \bar{\alpha} \in \cKi$.
To see this, let $x \in \mathbb{X}$, $u \in \ucccpi(x)$ and $n\geq 0$ be arbitrary, then
\begin{equation*}
q\big(\sigma(\phi(n;x,u))\big) \leq J_\infty(x,u\mid \ell) \leq \bar{\alpha}(\sigma(x)),
\end{equation*}
and applying $q^{-1}\in \cKi$ to both sides leads to
\begin{equation*}
\sigma(\phi(n;x,u)) \leq \left(q^{-1} \circ \bar{\alpha} \right)(\sigma(x)).
\end{equation*}
\textbf{Step 2} Claim: There exists $\tilde{\alpha} \in \cKi$ such that for all $R,\epsilon\in\Rp$ there exists $N=N(R,\epsilon)\in\mathbb{N}$ such that for all $x\in\mathbb{X}$ with $\sigma(x)\leq R$ there is $u=u(x,\epsilon)\in U^\infty$ with
\begin{description}
\item[(i)] $J_\infty(x,u\mid \ell) \leq \tilde{\alpha}(\sigma(x))$
\item[(ii)] $\forall n\geq 0: \phi(n;x,u)\in\mathbb{X}$
\item[(iii)] $\forall n \geq N: \sigma(\phi(n;x,u)) < \epsilon$
\end{description}
Proof of claim: For this we use an argument from \cite[Section~4.2]{TeelEtAl10}. Define $\tilde{\alpha} = \bar{\alpha} + \bar{\alpha} \circ \gamma_\sigma$.
Let $R,\epsilon \in\Rp$ be arbitrary, choose some $0<\tilde{\epsilon} < \epsilon$ and let $N=N(R,\epsilon)$ be the smallest positive integer with
\begin{equation*}
N \geq \frac{\bar{\alpha}(R)}{q\left(\gamma_\sigma^{-1}(\tilde{\epsilon})\right)}-1.
\end{equation*}
Let $x \in \mathbb{X}$ with $\sigma(x)\leq R$ be arbitrary and choose some $u_1\in\ucccpi(x)$.
Then there exists $n_q \in \{ 0,1,\ldots,N\}$ such that $q\big(\sigma(\phi(n_q))\big)\leq q\left(\gamma_\sigma^{-1}(\tilde{\epsilon})\right)$, where we defined $\phi(n_q)=\phi(n_q;x,u_1)$ to avoid notational clutter. Otherwise,
\begin{align*}
\bar{\alpha}(\sigma(x)) & \geq J_\infty(x,u_1\mid \ell) \geq \sum_{n=0}^N q(\sigma(\phi(n;x,u_1))) \\
& > (N+1)q\left(\gamma_\sigma^{-1}(\tilde{\epsilon})\right) \geq \bar{\alpha}(R) \geq \bar{\alpha}(\sigma(x)),
\end{align*}
a contradiction. 
Therefore we have $q\big(\sigma(\phi(n_q))\big)\leq q\left(\gamma_\sigma^{-1}(\tilde{\epsilon})\right)$ and
because $q \in \cKi$, also $\sigma(\phi(n;x,u_1))\leq \gamma_\sigma^{-1}(\tilde{\epsilon})$.
Since $u_1\in\ucccpi(x)$ ensures $\phi(n_q;x,u_1)\in\mathbb{X}$, we can choose some $u_2\in\ucccpi(\phi(n_q;x,u_1))$.
Define now
\begin{equation*}
u =
\begin{pmatrix}
u_1(0) & \cdots & u_1(n_q-1) & u_2(0) & u_2(1) & \cdots
\end{pmatrix},
\end{equation*} 
then
\begin{align*}
J_\infty(x,u\mid \ell) & = \sum_{n=0}^{n_q-1} \ell(\phi(n;x,u_1),u_1(n)) \\
& \hspace{0.5cm} +  \sum_{n=0}^\infty \ell(\phi(n; \phi(n_q),u_2), u_2(n)) \\
& \leq J_\infty(x,u_1 \mid \ell) + J_\infty(\phi(n_q),u_2 \mid \ell) \\
& \leq \bar{\alpha}(\sigma(x)) + \bar{\alpha}(\sigma(\phi(n_q))) \\
& \leq (\bar{\alpha} + \bar{\alpha} \circ \gamma_\sigma)(\sigma(x)) = \tilde{\alpha}(\sigma(x))
\end{align*}
where we used $\sigma(\phi(n_q))\leq \gamma_\sigma(\sigma(x))$ in the last line.
This establishes part (i) of the claim.
By construction, for all $n=0,\ldots,n_q$ we have $\phi(n;x,u)=\phi(n;x,u_1)\in\mathbb{X}$, and for $n\geq n_q$ also $\phi(n;x,u) = \phi(n-n_q;\phi(n_q),u_2) \in \mathbb{X}$, which shows part (ii) of the claim.
Finally, for all $n\geq n_q$ we have
\begin{align*}
\sigma(\phi(n;x,u)) & = \sigma(\phi(n-n_q;\phi(n_q),u_2)) \\
& \leq \gamma_\sigma\big(\sigma(\phi(n_q;x,u_1))\big) \leq \tilde{\epsilon} < \epsilon,
\end{align*}
and since $n_q \leq N(R,\epsilon)$, we get that for all $n\geq N(R,\epsilon)$ $\sigma(\phi(n;x,u))<\epsilon$, which finishes the proof of the claim.

\textbf{Step 3} Let $R \in \Rp$ be arbitrary, let $(\epsilon_n)_{n\in\mathbb{N}_+}$ be a sequence of real numbers such that 
$0 < \epsilon_m \leq 1$ and $\epsilon_m > \epsilon_{m+1}$ for all $m\geq 1$ and $\lim_{m\rightarrow \infty} \epsilon_m = 0$.
Define a new sequence by
\begin{equation*}
\tilde{\epsilon}_m := \min\{ (\tilde{\alpha}^{-1} \circ q)(\epsilon_m), \tilde{\alpha}^{-1}(2^{-m} \bar{\alpha}(R)), R\}.
\end{equation*}
Finally, define $M_n = \sum_{m=1}^n N(R,\tilde{\epsilon}_m)$.
Let $x \in \mathbb{X}$ with $\sigma(x)\leq R$ be arbitrary.
Define now
\begin{align*}
\phi_0 & = x \\
u_0 &= u(x,\tilde{\epsilon}_1)\lvert{\{0,\ldots,N(R,\tilde{\epsilon}_1)-1\}} \\
\phi_1 & = \phi(N(R,\tilde{\epsilon}_1);x,u_0)
\end{align*}
where $u(x,\tilde{\epsilon}_1)$ is the input sequence from Step 2.
This implies that
\begin{equation*}
\sigma(\phi_1)=\sigma(\phi(N(R,\tilde{\epsilon}_1);x,u_0)) < \tilde{\epsilon}_1 \leq R
\end{equation*}
and $\phi_1 \in \mathbb{X}$.
We can now proceed inductively. 
Define for $n\geq 1$
\begin{align*}
u_n &= u(\phi_n,\tilde{\epsilon}_{n+1})\lvert{\{0,\ldots,N(R,\tilde{\epsilon}_{n+1})-1\}} \\
\phi_{n+1} & = \phi(N(R,\tilde{\epsilon}_{n+1});\phi_n,u_n)
\end{align*}
and then $u = 
\begin{pmatrix}
u_0 & u_1 & \cdots 
\end{pmatrix}\in U^\infty$.
Claim: (i) For all $m \geq 1$ and $n\geq M_m$ we have $\sigma(\phi(n;x,u))< \epsilon_m$.
(ii) $J_\infty(x,u\mid \ell)\leq \hat{\alpha}(R)$, where  $\hat{\alpha} = \tilde{\alpha} + \bar{\alpha}$.

Proof of claim: Let $m\geq 1$ and $n\geq M_m$ be arbitrary. Let $m_\ast$ be the largest positive integer such that $n \geq M_{m_\ast}$, then
\begin{align*}
\sigma(\phi(n;x,u)) & = \sigma(\phi(n-M_{m_\ast};\phi_{m_\ast},u_{m_\ast})) \\
& \leq q^{-1}\left( J_\infty(\phi_{m_\ast}, u(\phi_{m_\ast},\tilde{\epsilon}_{m_\ast+1}) \mid \ell)\right) \\
& \leq q^{-1}\left(\tilde{\alpha}(\sigma(\phi_{m_\ast})) \right) \\
& < q^{-1}\left( \tilde{\alpha}(\tilde{\epsilon}_{m_\ast}) \right) \\
& \leq q^{-1}\left( \tilde{\alpha}\left( (\tilde{\alpha}^{-1} \circ q)(\epsilon_{m_\ast})\right) \right) \\
& \leq \epsilon_{m_\ast} \leq \epsilon_m,
\end{align*}
where we used 
$\ell\geq 0$ and the choice of $u_{m_\ast}$ in the first inequality, 
property (i) of $u_{m_\ast}$ established in Step 2 for the second inequality, 
the inductive definition of $\phi_{m_\ast}$ for the third inequality, 
the definition of $\tilde{\epsilon}_{m_\ast}$ in the fourth inequality 
and finally $m_\ast \geq m$. Altogether, this shows (i). Next,
\begin{align*}
J_\infty(x,u \mid \ell) & = \sum_{m=0}^\infty \sum_{n=0}^{N(R,\tilde{\epsilon}_{m+1})-1} \ell(\phi(n;\phi_m,u_m),u_m(n)) \\
& \leq  \sum_{m=0}^\infty J_\infty(\phi_m,u(\phi_m,\tilde{\epsilon}_{m+1}) \mid \ell) \\
& <  \sum_{m=0}^\infty \tilde{\alpha}(\tilde{\epsilon}_m) \\
& \leq \tilde{\alpha}(\sigma(x)) + \sum_{m=1}^\infty \tilde{\alpha}\left( \tilde{\alpha}^{-1}\left(2^{-m} \bar{\alpha}(R)\right) \right) \\
& \leq (\tilde{\alpha} + \bar{\alpha})(R) = \hat{\alpha}(R),
\end{align*}
establishing (ii).

\textbf{Step 4} Fix any sequence $(\epsilon_m)_{m \in \mathbb{N}_+}$ as required by Step 3 ($\epsilon_m=1/m$ works for example). 
For any $R \in \Rp$ define $N_R:\Rp\rightarrow \mathbb{N}$ by $N_R(\epsilon)=M_{m_\ast}$, where $m_\ast$ is the smallest integer such that $\epsilon_{m_\ast} < \epsilon$ (exists since $\epsilon_m \rightarrow 0$). 
Step 3 then shows that for all $x \in \mathbb{X}$ there exists $u_x \in U^\infty$ with $J_\infty(x,u_x\mid \ell) \leq \hat{\alpha}(\sigma(x))$ and for all $\epsilon \in \Rp$ and $n \geq N_{\sigma(x)}(\epsilon)$ also $\sigma(\phi(n;x,u_x))\leq \epsilon$ (simply set $R=\sigma(x)$ in Step 3 and use the input $u$ constructed there). 
Since
\begin{equation*}
\sum_{n=0}^\infty r\big(\rho(u_x(n))\big) \leq J_\infty(x,u_x \mid \ell) \leq \hat{\alpha}(\sigma(x)),
\end{equation*}
we can set $\eta=r$ and $\gamma=\hat{\alpha}$. Furthermore, by construction, for all $R,\epsilon \in \Rp$ we have $N_R(\epsilon) \in \mathbb{N}$, $N_R(\cdot)$ nonincreasing and hence Borel-measurable, $\tilde{R} \mapsto N_{\tilde{R}}(\epsilon)$ nondecreasing.

\textbf{Step 5} 
We finish the proof using a standard argument.
Define for $R,\epsilon \in \Rp$
\begin{equation*}
\nu_R(\epsilon):= \frac{2}{\epsilon}\int_{\epsilon/2}^\epsilon N_R(s)\mathrm{d}s + \frac{R}{\epsilon}.
\end{equation*}
Using Step 4, we see that $\nu$ is well-defined and continuous, we have $\nu(R,\epsilon) \geq N_R(\epsilon)$ and $\nu(R,\cdot)$ is strictly decreasing with $\lim_{\epsilon\rightarrow 0}\nu(R,\epsilon)=0$. Hence for each $R\in\Rp$ the inverse $\nu_R^{-1}$ exists, is continuous, strictly decreasing and for all $n\geq 1$ we have $n = \nu_R(\nu_R^{-1}(n)) \geq N_R(\nu_R^{-1}(n))$. This means that for all $x \in \mathbb{X}$ and the input sequence $u_x$ from Step 4 we have $\sigma(\phi(n;x,u_x)) \leq \nu_R^{-1}(n)$ for all $n \geq 1$.
Finally, using the construction from \cite[Section~2.7.3]{Tran19} and \cite[Section~C.6]{Khalil02} we can build a suitable $\beta \in \cKL$ from $\gamma_\sigma$ and $(R,\epsilon)\mapsto \nu_R^{-1}(\epsilon)$ such that for all $x \in \mathbb{X}$ and corresponding $u_x \in U^\infty$ from Step 4, we have $\sigma(\phi(n;x,u_x))\leq \beta(\sigma(x),n)$ for all $n\geq 0$.
\end{proof}
\begin{corollary} \label{cor.converse}
System \eqref{eq.sys} is UACwUBgEC on $X$ if and only if it is UCC on $X$ with respect to a stage cost $\ell(x,u)=q(\sigma(x))+r(\rho(u))$ with $q,r \in \cKi$.
\end{corollary}
An interesting feature of the proof of Theorem \ref{thm.converse} and hence of Corollary \ref{cor.converse} is that it is elementary (no auxiliary tools from Dynamic Programming or Lyapunov theory are used) and constructive (at least in principle explicit formulas are derived) nature.

\section{CONCLUSION} \label{sec.conclusion}
In this work, we generalized the constructive method described in \cite[Section~4.1]{GruenePannek17} for finding a stage cost for infinite horizon optimal control ensuring uniform cost controllability, and we investigated the required controllability conditions.
The generalization to the case of time-varying systems with corresponding uniformity assumptions is straightforward and has been omited due to space constraints. Furthermore, at least for summable nonnegative weighting sequences, the case of a total cost functional with weighting is also a straightforward extension of our results.
% TODO Future work
Interesting directions for further investigations include the development of more explicit techniques for manipulating comparison functions and corresponding numerical algorithms. Another interesting direction is the characterization of all stage costs with certain properties like prescribed growth rates or convexity, under given controllability assumptions.
\section*{ACKNOWLEDGEMENTS}
We thank L. Kreisk\"other, A. Gr\"afe, P.-F. Massiani and three anonymous reviewers for helpful comments.

\bibliographystyle{plain} 
\bibliography{fiedlertrimpe_med2022} 

\end{document}